\documentclass[pdftex,11pt]{amsart} 
\usepackage[pdftex]{hyperref} 
\usepackage{amsmath,amssymb}
\usepackage{mathtools}
\usepackage{amscd}
\usepackage{mathrsfs}
\usepackage[foot]{amsaddr}
\usepackage{amsthm}
\usepackage{setspace}
\usepackage[all]{xy}
\usepackage{comment}
\usepackage[pdftex]{graphicx} 
\usepackage{tikz}
\usetikzlibrary{arrows}
\usetikzlibrary{shapes.geometric,fit}
\usepackage{tikz-cd}
\usepackage{fancyhdr}

\theoremstyle{definition}
\newtheorem{thm}{Theorem}[subsection]
\newtheorem*{thm*}{Theorem}
\newtheorem{defi}[thm]{Definition}
\newtheorem*{defi*}{Definition}
\newtheorem*{acknowledge}{Acknowledgement}
\newtheorem{cor}[thm]{Corollary}
\newtheorem*{cor*}{Corollary}
\newtheorem{prop}[thm]{Proposition}
\newtheorem*{prop*}{Proposition}
\newtheorem{lem}[thm]{Lemma}
\newtheorem*{lem*}{Lemma}
\newtheorem{ex}[thm]{Example}
\newtheorem*{ex*}{Example}

\newtheorem{rem}[thm]{Remark}
\newtheorem*{rem*}{Remark}

\newtheorem*{hw*}{Homework}

\newcommand{\C}{\mathbb{C}}

\newcommand{\Z}{\mathbb{Z}}
\newcommand{\N}{\mathbb{N}}
\newcommand{\T}{\mathbb{T}}
\DeclareMathOperator{\supp}{supp}

\newcommand{\tight}{\mathrm{tight}}
\newcommand{\Bis}{\mathrm{Bis}}
\DeclareMathOperator{\Iso}{\mathrm{Iso}}

\DeclareMathOperator{\Span}{\mathrm{span}}
\DeclareMathOperator{\ev}{ev}

\DeclareMathOperator{\dom}{\mathrm{dom}}

\DeclareMathOperator{\id}{\mathrm{id}}

\def\i<#1>{\langle #1 \rangle}
\def\l<#1>{\left\langle #1 \right\rangle}

\makeatletter
\renewenvironment{proof}[1][\proofname]{\par
  \normalfont
  \topsep6\p@\@plus6\p@ \trivlist
  \item[\hskip\labelsep{\bfseries #1}\@addpunct{.}]\ignorespaces
}{%
  \endtrivlist
}
\renewcommand{\proofname}{\sc{Proof}}
\makeatother
\makeatletter
\newcommand*{\defeq}{\mathrel{\rlap{%
                     \raisebox{0.3ex}{$\m@th\cdot$}}%
                     \raisebox{-0.3ex}{$\m@th\cdot$}}%
                     =}
\makeatother

\title[]{Submodules of normalisers in groupoid C*-algebras and discrete group coactions}
\author{Fuyuta Komura}
\address{Center for Advanced Intelligence Project, RIKEN,
	3-14-1 Hiyoshi, Kohoku-ku, Yokohama, 223-8522, Japan}
\address{Phone: +81-45-566-1641+42706}
\address{Fax: +81-45-566-1642}
\email{fuyuta.k@keio.jp}

\begin{document}
\maketitle
\begin{abstract}
	
	In this paper, we investigate certain submodules in C*-algebras associated to effective \'etale groupoids.
	First, we show that a submodule generated by normalizers is a closure of the set of compactly supported continuous functions on some open set.
	As a corollary, we show that discrete group coactions on groupoid C*-algebras are induced by cocycles of \'etale groupoids if the fixed point algebras contain C*-subalgebras of continuous functions vanishing at infinity on the unit spaces.
	In the latter part, we prove the Galois correspondence result for discrete group coactions on groupoid C*-algebras.

\end{abstract}

\section{Introduction}

\'Etale groupoids are widely utilized in the theory of C*-algebras.
Since Renault initiated the theory of groupoid C*-algebras in \cite{renault1980groupoid},
many researchers have studied groupoid C*-algebras.
One of the natural tasks is to characterize properties of groupoid C*-algebras in terms of \'etale groupoids.
For example,
the relation between nuclearity of groupoid C*-algebras and amenability of topological groupoids is studied in \cite{anantharaman2000}.
In addition, simplicity of full groupoid C*-algebras is characterized in \cite{Brown2014}.
Recently, the authors in \cite{BrownExelFuller2021} established the Galois correspondence result between \'etale groupoids and twisted groupoid C*-algebras.
Namely,
the authors in \cite{BrownExelFuller2021} proved that there is a one-to-one correspondence between open subgroupoids of an \'etale groupoid $G$ and certain intermediate C*-subalgebras of twisted groupoid C*-algebras that contain $C_0(G^{(0)})$,
provided that the underlying \'etale groupoid $G$ is effective.
In the present paper,
we study certain $C_0(G^{(0)})$-submodules in groupoid C*-algebras.
More precisely,
we prove that there exists a one-to-one correspondence between $C_0(G^{(0)})$-submodules of normalizers in groupoid C*-algebras and open sets in \'etale groupoids under the assumption that \'etale groupoids are effective (Theorem \ref{theorem main theorem about sub module}).
This result allows us to capture $C_0(G^{(0)})$-submodules by open sets in \'etale groupoids.
We have some applications of this result as follows.

Submodules in C*-algebras appear in many contexts.
We give such examples relevant to the present paper.
First,
we consider discrete group coactions on C*-algebras.
Discrete group coactions generalize compact abelian group actions on C*-algebras and are studied in 	\cite{quigg_1996}, for example.
Assume that $\Gamma$ is a discrete group,
$A$ is a C*-algebra and $\delta\colon A\to A\otimes C^*_r(\Gamma)$ is a coaction.
Then one can decompose $A$ into the spectral submodules as $A=\overline{\sum_{s\in\Gamma}A_s}$,
where $A_s\subset A$ is the spectral submodule of $s\in \Gamma$.
Note that $A_s$ is a $A_e$-submodule in $A$,
where $A_e$ is the fixed point algebra of the coaction $\delta$.
Hence it is essential to understand spectral submodules in the study of coactions.
In the present paper,
we consider discrete group coactions on groupoid C*-algebras $\delta\colon C^*_r(G)\to C^*_r(G)\otimes C^*_r(\Gamma)$.
We prove that, if the fixed point algebra $C^*_r(G)_e$ contains $C_0(G^{(0)})$,
then the coaction $\delta\colon C^*_r(G)\to C^*_r(G)\otimes C^*_r(\Gamma)$ is induced from a continuous cocycle $c\colon G\to \Gamma$ (Corollary \ref{corollary coaction of discrete group on groupoid C*-algebras}).
Many problems about coactions are reduced to problems about continuous cocycles $c\colon G\to \Gamma$,
which are more or less treatable.
For example,
we prove the Galois correspondence result for discrete group coactions on groupoid C*-algebras (Theorem \ref{theorem intermediate subgroupoids and subgroups}).
We mention that similar results are obtained in \cite[Theorem A.1]{BrownFullerPitts2021}.
They assumed that the linear span of homogeneous normalisers is dense in the whole algebra.
It seems worth noting that we do not impose such an assumption in Corollary \ref{corollary coaction of discrete group on groupoid C*-algebras} (instead, we assume that \'etale groupoids are effective).

Submodules in C*-algebras also appear in noncommutative Cartan subalgebras \cite{noncommutativeCartanExel}.
For an inclusion of C*-algebras $B\subset A$,
Exel defined an inverse semigroup that consists of slices,
which are certain $B$-submodules in $A$.
In this paper,
we prove that the inverse semigroups of slices are isomorphic to the inverse semigroups of bisections in $G$ if we take $A=C^*_r(G)$ and $B=C_0(G^{(0)})$, where $G$ is an effective \'etale groupoid (Corollary \ref{corollary inverse semigroup of bisections are inverse semigroup of slices}).

This paper is organized as follows.
In Section 1,
we recall fundamental notions on \'etale groupoids,
groupoid C*-algebras, discrete group coactions and inverse semigroups.
In Section 2, we prove our main theorems.
First,
we prove that there exists a one-to-one correspondence between $C_0(G^{(0)})$-submodules of normalizers in a groupoid C*-algebra $C^*_r(G)$ and open sets in the underlying \'etale groupoid $G$ (Theorem \ref{theorem main theorem about sub module}).
Then we obtain two corollaries.
The first corollary deals with inverse semigroups of slices.
More precisely, we prove that inverse semigroups of slices are isomorphic to inverse semigroups of bisections for groupoid C*-algebras associated to effective \'etale groupoids (Corollary \ref{corollary inverse semigroup of bisections are inverse semigroup of slices}).
The second corollary deals with discrete group coactions on groupoid C*-algebras.
Namely, we obtain characterization of discrete group coactions on groupoid C*-algebras that are induced from cocycles (Corollary \ref{corollary coaction of discrete group on groupoid C*-algebras}).
In Subsection \ref{subsection Galois correspondence for discrete group coactions on groupoid C*-algebras },
we obtain the Galois correspondence result for discrete group coactions on groupoid C*-algebras (Corollary \ref{corollary galois corr for discrete group coactions}) by analysing discrete group cocycles.

\begin{acknowledge}
	The author would like to thank Prof.\ Takeshi Katsura for his support and encouragement.
	The author is also grateful to Yuhei Suzuki for fruitful discussions.
	In addition, the author would like to express his gratitude to the anonymous reviewer for the constructive comments and suggestions.
	This work was supported by JST CREST Grant Number JPMJCR1913 and RIKEN Special Postdoctoral Researcher Program.
\end{acknowledge}

\section{Preliminaries}

In this section,
we recall fundamental notions about \'etale groupoids and groupoid C*-algebras.
In the last half of this section,
we recall inverse semigroups and discrete group coactions on C*-algebras.

\subsection{\'Etale groupoids}

We recall the basic notions on \'etale groupoids.
See \cite{asims} and \cite{paterson2012groupoids} for more details.

A groupoid is a set $G$ together with a distinguished subset $G^{(0)}\subset G$,
domain and range maps $d,r\colon G\to G^{(0)}$ and a multiplication 
\[
G^{(2)}\defeq \{(\alpha,\beta)\in G\times G\mid d(\alpha)=r(\beta)\}\ni (\alpha,\beta)\mapsto \alpha\beta \in G
\]
such that
\begin{enumerate}
	\item for all $x\in G^{(0)}$, $d(x)=x$ and $r(x)=x$ hold,
	\item for all $\alpha\in G$, $\alpha d(\alpha)=r(\alpha)\alpha=\alpha$ holds,
	\item for all $(\alpha,\beta)\in G^{(2)}$, $d(\alpha\beta)=d(\beta)$ and $r(\alpha\beta)=r(\alpha)$ hold,
	\item if $(\alpha,\beta),(\beta,\gamma)\in G^{(2)}$,
	we have $(\alpha\beta)\gamma=\alpha(\beta\gamma)$,
	\item\label{inverse} every $\gamma \in G$,
	there exists $\gamma'\in G$ which satisfies $(\gamma',\gamma), (\gamma,\gamma')\in G^{(2)}$ and $d(\gamma)=\gamma'\gamma$ and $r(\gamma)=\gamma\gamma'$.   
\end{enumerate}
Since the element $\gamma'$ in (\ref{inverse}) is uniquely determined by $\gamma$,
$\gamma'$ is called the inverse of $\gamma$ and denoted by $\gamma^{-1}$.
We call $G^{(0)}$ the unit space of $G$.
A subgroupoid of $G$ is a subset of $G$ which is closed under the inversion and multiplication. 
A subgroupoid of $G$ is said to be wide if it contains $G^{(0)}$.
For $U\subset G^{(0)}$, we define $G_U\defeq d^{-1}(U)$ and $G^{U}\defeq r^{-1}(U)$.
We also define $G_x\defeq G_{\{x\}}$ and $G^x\defeq G^{\{x\}}$ for $x\in G^{(0)}$.

A topological groupoid is a groupoid equipped with a topology where the multiplication and the inverse are continuous.
A topological groupoid is said to be \'etale if the domain map is a local homeomorphism.
Note that the range map of an \'etale groupoid is also a local homeomorphism.
An \'etale groupoid is said to be ample if it has an open basis which consists of compact sets.
\textbf{In this paper,
we assume that \'etale groupoids are always locally compact Hausdorff unless otherwise stated}, although there exist important \'etale groupoids that are not Hausdorff.
In this paper, we mean locally compact Hausdorff \'etale groupoids by \'etale groupoids.

A subset $U$ of an \'etale groupoid $G$ is called a bisection if the restrictions $d|_U$ and $r|_U$ are injective.
It follows that $d|_U$ and $r|_U$ are homeomorphism onto their images if $U$ is a bisection since $d$ and $r$ are open maps.

An \'etale groupoid $G$ is said to be effective if $G^{(0)}$ coincides with the interior of $\Iso(G)$,
where 
\[
\Iso(G)\defeq\{\alpha\in G\colon d(\alpha)=r(\alpha)\}
\]
is the isotropy of $G$.
An \'etale groupoid $G$ is said to be topologically transitive if $r(d^{-1}(U))$ is dense in $G^{(0)}$ for all non-empty open set $U\subset G^{(0)}$.
If $G$ has a dense orbit (i.e.\ there exists $x\in G^{(0)}$ such that $r(d^{-1}(\{x\}))$ is dense in $G^{(0)}$),
then $G$ is topologically transitive.
The converse is true if $G^{(0)}$ is second countable (see \cite[Lemma 3.4]{STEINBERG20192474}).
An \'etale groupoid $G$ is minimal if $r(d^{-1}(\{x\}))$ is dense in $G^{(0)}$ for all $x\in G^{(0)}$.
Obviously, a minimal \'etale groupoid is topologically transitive.

\subsection{Groupoid C*-algebras}

We recall the definition of groupoid C*-algebras.

Let $G$ be an \'etale groupoid.
Then $C_c(G)$, the vector space of compactly supported continuous functions on $G$, is a *-algebra with respect to the multiplication and the involution defined by
\[
f*g(\gamma)\defeq\sum_{\alpha\beta=\gamma}f(\alpha)g(\beta), f^*(\gamma)\defeq\overline{f(\gamma^{-1})},
\]
where $f,g\in C_c(G)$ and $\gamma\in G$.
The left regular representation $\lambda_x\colon C_c(G)\to B(\ell^2(G_x))$ at $x\in G^{(0)}$ is defined by
\[
\lambda_x(f)\delta_{\alpha}\defeq \sum_{\beta\in G_{r(\alpha)}}f(\alpha)\delta_{\alpha\beta},
\]
where $f\in C_c(G)$ and $\alpha\in G_x$.
The reduced norm $\lVert\cdot\rVert_{r}$ on $C_c(G)$ is defined by
\[
\lVert f\rVert_r\defeq \sup_{x\in G^{(0)}} \lVert \lambda_x(f)\rVert
\]
for $f\in C_c(G)$.
We often omit the subscript `$r$' of $\lVert\cdot\rVert_r$ if there is no chance to confuse.
The reduced groupoid C*-algebra $C^*_r(G)$ is defined to be the completion of $C_c(G)$ with respect to the reduced norm.
Note that $C_c(G^{(0)})\subset C_c(G)$ is a *-subalgebra and this inclusion extends to the inclusion  $C_0(G^{(0)})\subset C^*_r(G)$.
The next proposition is well-known.

\begin{prop}[Evaluation] \label{prop evaluation}
	Let $G$ be an \'etale groupoid and $a\in C^*_r(G)$.
	Then $j(a)\in C_0(G)$ is defined by
	\[
	j(a)(\alpha)\defeq\i<\delta_{\alpha}|\lambda_{d(\alpha)}(a)\delta_{d(\alpha)}>
	\]
	for $\alpha\in G$\footnote{In this paper, inner products of Hilbert spaces are linear with respect to the right variables.}.
	Then $j\colon C^*_r(G)\to C_0(G)$ is a norm decreasing injective linear map.
	Moreover, $j$ is an identity map on $C_c(G)$.
	
\end{prop}

\begin{rem}
	Since $j\colon C^*_r(G)\to C_0(G)$ is injective,
	we may identify $j(a)$ with $a$.
	Hence, we often regard $a$ as a function on $G$ and simply denote $j(a)$ by $a$.
\end{rem}

	For $a\in C^*_r(G)$,
	we denote the open support of $a$ by
	\[
	\supp^{\circ}(a)\defeq \{\alpha\in G\mid a(\alpha)\not=0\}.
	\]
	Note that $\supp^{\circ}(a)$ is open in $G$.

\begin{defi}
	Let $A$ be a C*-algebra and $D\subset A$ be a C*-subalgebra.
	An element $n\in A$ is called a normalizer for $D$ if $nDn^*\cup n^*Dn\subset D$ holds.
	We denote the set of normalizers for $D$ by $N(A, D)$.
	
\end{defi}	

\begin{prop}[{\cite[Proposition 4.8]{renault}}] \label{prop open support of normaliser is a bisection}
	Let $G$ be an \'etale groupoid and $U\subset G$ be an open set.
	If $U$ is a bisection, then every elements in $C_c(U)$ is a normalizer for $C_0(G^{(0)})$.
	Conversely, if $n\in C^*_r(G)$ is a normalizer for $C_0(G^{(0)})$ and $G$ is effective,
	then $\supp^{\circ}(n)\subset G$ is an open bisection.
\end{prop}

We will use the approximation of normalisers by intertwiners.

\begin{prop}[{\cite[Proposition 3.4]{DonsigPitts2008}}] \label{prop approximation of normalisers by intertwiners}
	Let $A$ be a C*-algebra and $D\subset A$ be a commutative C*-subalgebra.
	An element $n\in A$ is called an intertwiner for $D$ if $nD=Dn$ holds.
	Let $\mathcal{I}$ be the set of intertwiners for $D$.
	If $D\subset A$ is a maximal abelian subalgebra,
	then $N(A, D)=\overline{\mathcal{I}}$.
\end{prop}

	\subsection{Inverse semigroups}

	We recall the basic notions about inverse semigroups,
	although we will not often utilize inverse semigroups.
	We mainly utilize inverse semigroups to construct examples of \'etale groupoids.
	See \cite{lawson1998inverse} or \cite{paterson2012groupoids} for more details.
	An inverse semigroup $S$ is a semigroup where for every $s\in S$ there exists a unique $s^*\in S$ such that $s=ss^*s$ and $s^*=s^*ss^*$.
	We denote the set of all idempotents in $S$ by $E(S)\defeq\{e\in S\mid e^2=e\}$.
	It is known that $E(S)$ is a commutative subsemigroup of $S$.
	An inverse semigroup which consists of idempotents is called a (meet) semilattice of idempotents.
	A zero element is a unique element $0\in S$ such that $0s=s0=0$ holds for all $s\in S$.
	An inverse semigroup with a unit is called an inverse monoid.
	By a subsemigroup of $S$,
	we mean a subset of $S$ that is closed under the product and inverse of $S$.
	
	For a topological space $X$,
	we denote by $I_X$ the set of all homeomorphisms between open sets in $X$.
	Then $I_X$ is an inverse semigroup with respect to the product defined by the composition of maps.
	For an inverse semigroup $S$,
	an inverse semigroup action $\alpha\colon S\curvearrowright X$ is a semigroup homomorphism $S\ni s\mapsto \alpha_s\in I_X$.
	In this paper, we always assume that every action $\alpha$ satisfies $\bigcup_{e\in E(S)}\dom(\alpha_e)=X$.
	
	\subsection{Inverse semigroups associated to inclusions of C*-algebras}
		
	Following \cite[Proposition 13.3]{noncommutativeCartanExel},
	we define inverse semigroups of slices.
	
\begin{defi}
	Let $D\subset A$ be an inclusion of C*-algebras.
	A slice  is a norm closed subspace $M\subset A$ such that $DM\cup MD\subset M$ and $M\subset N(A, D)$.
	The set of all slices is denoted by $\mathcal{S}(A,D)$.
\end{defi}

\begin{prop}{{\cite[Proposition 13.3]{noncommutativeCartanExel}}}
	Let $D\subset A$ be an inclusion of C*-algebras.
	Assume that $D$ has an approximate unit for $A$.
	For $M,N\in \mathcal{S}(A,D)$,
	define $MN$ to be the closure of the linear span of
	\[
	\{xy\in A\mid x\in M, y\in N \}.
	\]
	Then $\mathcal{S}(A,B)$ is an inverse semigroup under this operation.
	The generalized inverse of $M\in \mathcal{S}(A, D)$ is $M^*\defeq\{x^*\in A\mid x\in M\}$.
\end{prop}

Let $G$ be an \'etale groupoid and $\mathrm{Bis}(G)$ denote the set of all open bisections in $G$.
For $U, V\in \mathrm{Bis}(G)$,
their product is defined by
\[
UV\defeq \{\alpha\beta\in G\mid \alpha\in U, \beta\in V, d(\alpha)=r(\beta)\}.
\]
Then $UV\in\mathrm {Bis}(G)$ and $\mathrm{Bis}(G)$ is an inverse semigroup.
Note that $U^*\in \mathrm{Bis}(G)$ is given by
\[
U^{-1}\defeq \{\alpha^{-1}\in G\mid \alpha\in U\}.
\]

Although the next result seems to be well-known,
we include a proof for the reader's convenience.

\begin{prop}\label{prop homomorphism from Bis to slices}
	Let $G$ be an \'etale groupoid.
	Then the map 
	\[\Psi\colon \mathrm{Bis}(G)\ni U\mapsto \overline{C_c(U)}\in \mathcal{S}(C^*_r(G), C_0(G^{(0)})) \]
	is an injective semigroup homomorphism.
\end{prop}

\begin{proof}
	By Proposition \ref{prop open support of normaliser is a bisection},
	$\Psi$ is well-defined.
	First, we show $\overline{C_c(UV)}=\overline{C_c(U)}\cdot\overline{C_c(V)}$ for $U,V\in \Bis(G)$.
	One can check $\overline{C_c(U)}\cdot \overline{C_c(V)}\subset \overline{C_c(UV)}$ by direct calculations.
	Since the reduced norm and the supremum norm coincide on $\overline{C_c(W)}$ for all $W\in \Bis(G)$,
	we obtain $\overline{C_c(U)}\cdot \overline{C_c(V)}=\overline{C_c(UV)}$ by applying Stone-Weierstrass theorem.
	It is clear that $\Psi$ preserves the inverses.\footnote{We need not check that $\Psi$ preserves the inverses, since semigroup homomorphisms between inverse semigroups automatically preserve the inverses.}
	
	Next, we show that $\Psi$ is injective.
	Assume that $U, V\in \Bis(G)$ are different.
	We may assume that there exists $\gamma\in U$ such that $\gamma\not\in V$.
	Consider the map
	\[
	\ev_{\gamma}\colon C^*_r(G)\ni a\mapsto a(\gamma) \in\C.
	\]
	Then $\ev_{\gamma}$ is not zero on $\overline{C_c(U)}$ by Urysohn's lemma.
	On the other hand,
	$\ev_{\gamma}$ is zero on $\overline{C_c(V)}$.
	Hence, $\Psi(U)\not=\Psi(V)$ and we have shown that $\Psi$ is injective.
	\qed
 \end{proof}

\subsection{\'Etale groupoids associated to inverse semigroup actions}

Many \'etale groupoids arise from actions of inverse semigroups to topological spaces.
We recall how to construct an \'etale groupoid from an inverse semigroup action.

Let $X$ be a locally compact Hausdorff space.
Recall that $I_X$ is the inverse semigroup of homeomorphisms between open sets in $X$.
For $e\in E(S)$, we denote the domain of $\alpha_e$ by $D_e^{\alpha}$.
Then $\alpha_s$ is a homeomorphism from $D_{s^*s}^{\alpha}$ to $D_{ss^*}^{\alpha}$.
We often omit $\alpha$ of $D_{e}^{\alpha}$ if there is no chance to confuse.

For an action $\alpha\colon S\curvearrowright X$,
we associate an \'etale groupoid $S\ltimes_{\alpha}X$ as the following.
First we put the set $S*X\defeq \{(s,x) \in S\times X \mid x\in D^{\alpha}_{s^*s}\}$.
Then we define an equivalence relation $\sim$ on $S*X$ by declaring that $(s,x)\sim (t,y)$ holds if
\[
\text{$x=y$ and there exists $e\in E(S)$ such that $x\in D^{\alpha}_e$ and $se=te$}.  
\]
Set $S\ltimes_{\alpha}X\defeq S*X/{\sim}$ and denote the equivalence class of $(s,x)\in S*X$ by $[s,x]$.
The unit space of $S\ltimes_{\alpha}X$ is $X$, where $X$ is identified with the subset of $S\ltimes_{\alpha}X$ via the injective map
\[
X\ni x\mapsto [e,x] \in S\ltimes_{\alpha}X, x\in D^{\alpha}_e.
\]
The domain map and range maps are defined by
\[
d([s,x])=x, r([s,x])=\alpha_s(x)
\]
for $[s,x]\in S\ltimes_{\alpha}X$.
The product of $[s,\alpha_t(x)],[t,x]\in S\ltimes_{\alpha}X$ is $[st,x]$.
The inverse should be $[s,x]^{-1}=[s^*,\alpha_s(x)]$.
Then $S\ltimes_{\alpha}X$ is a groupoid in these operations.
For $s\in S$ and an open set $U\subset D_{s^*s}^{\alpha}$,
define 
\[[s, U]\defeq \{[s,x]\in S\ltimes_{\alpha}X\mid x\in U\}.\]
These sets form an open basis of $S\ltimes_{\alpha}X$.
In these structures,
$S\ltimes_{\alpha}X$ is a locally compact \'etale groupoid,
although $S\ltimes_{\alpha}X$ is not necessarily Hausdorff.
In this paper,
we only treat an inverse semigroup action $\alpha\colon S\curvearrowright X$ such that $S\ltimes_{\alpha}X$ becomes Hausdorff.

\subsection{Discrete group coactions on C*-algebras}\label{subsection Discrete group coactions on C*-algebras}

We recall the fundamental notions about discrete group coactions on C*-algebras.
Let $A$ be a C*-algebra and $\Gamma$ be a discrete group.
We denote the comultiplication of $C^*_r(\Gamma)$ by $\delta_{\Gamma}\colon C^*_r(\Gamma)\to C^*_r(\Gamma)\otimes C^*_r(\Gamma)$.
Note that we consider the minimal tensor product here and $\delta_{\Gamma}$ is the *-homomorphism such that $\delta_{\Gamma}(\lambda_s)=\lambda_s\otimes\lambda_s$ for all $s\in\Gamma$.
A coaction of $\Gamma$ on $A$ is a nondegenerate *-homomorphism $\delta\colon A\to A\otimes C^*_r(\Gamma)$ such that $(\id\otimes \delta_{\Gamma})\circ \delta=(\delta\otimes\id)\circ\delta$.
For $s\in\Gamma$,
\[A_s\defeq\{a\in A\mid \delta(a)=a\otimes\lambda_s\}\]
is called the spectral subspace for $s\in \Gamma$.
Then $A_s$ is a closed $A_e$-sub-bimodule in $A$.
In particular,
$A_e$ is a C*-subalgebra of $A$ and called the fixed point algebra of $\delta$. 
By \cite[Lemma 1.5]{quigg_1996},
we have $A=\overline{\Span} \bigcup_{s\in\Gamma}A_s$.
Moreover,
for $s\in\Gamma$,
define $\Phi_s\colon A\to A_s$ by $\Phi_s(a)=(\id\otimes \mathrm{tr})((1\otimes \lambda_{s^{-1}}) \delta(a))$ for $a\in A$,
where $\mathrm{tr}\in C^*_r(\Gamma)^*$ denotes the canonical trace of $C^*_r(\Gamma)$.
Then $\Phi_s$ is a norm decreasing $A_e$-bimodule map such that $\Phi_s|_{A_s}=\id$.
In particular, $\Phi_e$ is a conditional expectation.

We recall discrete group coactions on groupoid C*-algebras induced by cocycles.
See Appendix A of \cite{BrownFullerPitts2021} for more details.
Let $G$ be an \'etale groupoid, $\Gamma$ be a discrete group and $c\colon G\to\Gamma$ be a continuous cocycle (i.e.\ continuous groupoid homomorphism).
Then, by \cite[Lemma 6.1]{CARLSEN2021107923}, there exists a *-homomorphism $\delta_c\colon C^*_r(G)\to C^*_r(G)\otimes C^*_r(\Gamma)$ such that, if $f\in C_c(c^{-1}(s))$ holds for some $s\in\Gamma$, then $\delta_c(f)=f\otimes \lambda_s$.
Such a *-homomorphism $\delta_c$ is clearly unique and called the coaction of $\Gamma$ on $C^*_r(G)$ associated with the cocycle $c\colon G\to\Gamma$.

\section{Main theorems}

In this section,
we prove our main theorems.
In the first subsection,
we investigate $C_0(G^{(0)})$-submodules in $C^*_r(G)$.
Precisely,
assuming that $G$ is effective, we prove that a $C_0(G^{(0)})$-submodule $M\subset C^*_r(G)$ is of the form $\overline{C_0(U)}$ if and only if $M$ is a closed linear span of normalizers (Theorem \ref{theorem main theorem about sub module}).
Then we investigate inverse semigroups of slices (Corollary \ref{corollary inverse semigroup of bisections are inverse semigroup of slices}).
Finally, we prove that every coaction on $C^*_r(G)$ whose fixed point algebra contains $C_0(G^{(0)})$ is induced from a cocycle.

In the last subsection,
we prove the Galois correspondence result for discrete group coactions on groupoid C*-algebras (Corollary \ref{corollary galois corr for discrete group coactions}).

\subsection{$C_0(G^{(0)})$-submodules in $C_r^*(G)$}

In this subsection,
we investigate $C_0(G^{(0)})$-modules in $C^*_r(G)$.

\begin{lem}\label{lemma approximate normalizers by functions}
	Let $G$ be an effective \'etale groupoid.
	Assume that $m\in C^*_r(G)$ is a normalizer for $C_0(G^{(0)})$ and $\varepsilon>0$.
	Then there exists $f\in C_0(G^{(0)})$ such that
	\begin{itemize}
		\item $\lVert fm-m\rVert\leq \varepsilon$, and
		\item $fm\in C_c(\supp^{\circ}(m))$.
	\end{itemize}
	
\end{lem}

\begin{proof}
	In this proof,
	we denote the reduced norm of $C^*_{r}(G)$ (resp.\ the supremum norm of $C_0(G^{(0)})$) by $\lVert\cdot\rVert_{r}$ (resp.\ $\lVert\cdot\rVert_{\infty}$).
	Put 
	\[
	K\defeq	\{\alpha\in\supp^{\circ}(m)\mid \lvert m(\alpha)\rvert\geq \varepsilon/\sqrt{2}\}.
	\]
	Then $K$ is a compact set.
	Using Urysohn's lemma,
	take $f\in C_c(r(\supp^{\circ}(m)))$ such that $f|_{r(K)}=1$ and $0\leq f\leq 1$.
	Then we have
	\[
	\lVert fmm^*-mm^*\rVert_{\infty}\leq \varepsilon^2/2.
	\]
	Indeed, if $x\in r(K)$,
	then $\lvert f mm^*(x)-mm^*(x)\rvert=0$.
	If $x\not\in r(K)$,
	then we have
	\begin{align*}
	mm^*(x)=
	\begin{cases}
	\lvert m(\alpha)\rvert^2 & (G^x\cap \supp^{\circ}(m)=\{\alpha\}), \\
	0 & (\text{otherwise}).
	\end{cases}
	\end{align*}
	Here, since $\supp^{\circ}(m)$ is a bisection by Proposition \ref{prop open support of normaliser is a bisection},
	$G^x\cap \supp^{\circ}(m)$ has one element at most and we denote it by $\alpha$.
	Since $\alpha\not\in K$ if it exists,
	we have  $mm^*(x)<\varepsilon^2/2$.
	Using this,
	we have $\lvert fmm^*(x)-mm^*(x)\rvert<\varepsilon^2/2$.
	Hence we obtain $\lVert fmm^*-mm^*\rVert_{\infty}\leq \varepsilon^2/2$.
	
	We have
	\begin{align*}
	\lVert fm-m\rVert_{r}^2&=\lVert (fm-m)(fm-m)^*\rVert_{r} \\
	&=\lVert fmm^*f-fmm^*-mm^*f+mm^*\rVert_{r}\\
	&\leq \lVert fmm^*f-fmm^*\rVert_{r}+\lVert mm^*-mm^*f\rVert_{r} \\
	&=\lVert fmm^*f-fmm^*\rVert_{\infty}+\lVert mm^*-mm^*f\rVert_{\infty}\leq \varepsilon^2.
	\end{align*}
	Here, we used the fact that the reduced norm and the supremum norm coincide with each others on $C_0(G^{(0)})$ in the last equality.
	
	Now, it remains to show $fm\in C_c(\supp^{\circ}(m))$.
	Since \[fm(\alpha)=f(r(\alpha))m(\alpha)\] holds for all $\alpha\in G$,
	we have $\supp (fm) \subset r|_{\supp^{\circ}(m)}^{-1}(\supp (f))$.
	Since $r|_{\supp^{\circ}(m)}$ is a homeomorphism onto the image and $\supp (f)$ is compact,
	$\supp (fm)$ is also compact subset of $\supp^{\circ}(m)$.
	Hence we have $fm\in C_c(\supp^{\circ}(m))$.
	\qed
\end{proof}

Now, we are ready to prove the first main theorem in this paper.

\begin{thm}\label{theorem main theorem about sub module}
	Let $G$ be an effective \'etale groupoid.
	Assume that $M\subset C^*_r(G)$ is a closed left $C_0(G^{(0)})$-submodule.
	Then there exists an open set $U\subset G$ such that $M=\overline{C_c(U)}$ if and only if $M$ is a closed linear span of normalizers.
\end{thm}

\begin{rem}
	We remark that $\overline{C_c(U)}$ is a $C_0(G^{(0)})$-bimodule.
	Hence, in the above situation, 
	if a closed left $C_0(G^{(0)})$-submodule $M\subset C^*_r(G)$ is a closed linear span of normalizers,
	then $M$ automatically becomes a $C_0(G^{(0)})$-bimodule.
\end{rem}
\begin{proof}[\sc{Proof of Theorem \ref{theorem main theorem about sub module}}]
	Let $U\subset G$ be an open set.
	Since a family of the open bisections is an open basis of $G$,
	one can see that $C_c(U)$ is a linear span of normalizers by a partition of unity argument and Proposition \ref{prop open support of normaliser is a bisection}.
	Hence, $\overline{C_c(U)}$ is a closed linear span of normalizers.
	
	Assume that $M\subset C^*_r(G)$ is a closed linear span of normalizers.
	It is sufficient to show that there exists an open set $U\subset G$ such that $M=\overline{C_c(U)}$.
	Put 
	\[
	U\defeq \{\alpha\in G\mid \text{$a(\alpha)\not=0$ for some $a\in M$}\}.
	\]
	Then  we have
	\[
	U=\bigcup_{a\in M}\supp^{\circ}(a)=\bigcup_{m\in M\cap N(C^*_r(G), C_0(G^{(0)}))}\supp^{\circ}(m),
	\]
	where we used the assumption that $M$ is generated by normalizers in the last equation.
	In particular, $U\subset G$ is an open subset.
	We claim that $M=\overline{C_c(U)}$.
	
	First, we show that $M\subset \overline{C_c(U)}$.
	It suffices to show $m\in \overline{C_c(U)}$ for all normalizers $m$ in $M$.
	Take $\varepsilon>0$.
	Using Lemma \ref{lemma approximate normalizers by functions},
	take $f\in C_0(G^{(0)})$ such that $\lVert fm-m\rVert<\varepsilon$ and $fm\in C_c(\supp^{\circ}(m))$.
	Since $\supp^{\circ}(m)\subset U$,
	it follows that $fm\in C_c(U)$.
	Hence $m\in \overline{C_c(U)}$.
	
	Next, we show that $\overline{C_c(U)}\subset M$.
	It suffices to show $C_c(\supp^{\circ}(m))\subset M$ for each normalizer $m\in M$,
	since one can show
	\[
	C_c(U)=\Span_{m\in M\cap N(C^*_r(G), C_0(G^{(0)}))} C_c(\supp^{\circ}(m))
	\]
	by a partition of unity argument.
	Take $f\in C_c(\supp^{\circ}(m))$.
	Since $\supp^{\circ}(m)$ is a bisection by Proposition \ref{prop open support of normaliser is a bisection},
	$r|_{\supp^{\circ}(m)}$ is a homeomorphism onto the image.
	Define $g\in C_c(G^{(0)})$ by
	\begin{align*}
		g(x)=
		\begin{cases}
			f(r|_{\supp^{\circ}(m)}^{-1}(x))/m(r|_{\supp^{\circ}(m)}^{-1}(x))  & (x\in r(\supp^{\circ}(m))),\\
			0 & (\text{otherwise}).
		\end{cases}
	\end{align*}
	Then one can check that $f=gm$.
	Since $M$ is a left $C_0(G^{(0)})$-submodule,
	$f=gm$ belongs to $M$.
	\qed
\end{proof}

Let $G$ be an \'etale groupoid.
An intermediate C*-subalgebra $C_0(G^{(0)})\subset B\subset C^*_r(G)$ is obviously a $C_0(G^{(0)})$-bimodule.
In addition, for an open subgroupoid $G^{(0)}\subset H\subset G$,
$C^*_r(H)$ naturally becomes a C*-subalgebra of $C^*_r(G)$ (\cite[Lemma 3.2]{BrownExelFuller2021}).
If we apply Theorem \ref{theorem main theorem about sub module} to intermediate C*-subalgebras $C_0(G^{(0)})\subset B\subset C^*_r(G)$,
we obtain the next corollary,
which has already been obtained in \cite[Theorem 3.3]{BrownExelFuller2021}.

\begin{cor}[{\cite[Theorem 3.3]{BrownExelFuller2021}}]\label{corollary Galois corr between subgroupoid and subalgebra}
	Let $G$ be an effective \'etale groupoid.
	Then the map $H\mapsto C^*_r(H)$ is a bijection between the set of open subgroupoids $G^{(0)}\subset H\subset G$ and the set of C*-subalgebras $C_0(G^{(0)})\subset B\subset C^*_r(G)$ which are closed linear spans of normalisers.
\end{cor}
\begin{proof}
	It is sufficient to show that the map $H\mapsto C^*_r(H)$ is surjective.
	Assume that a C*-subalgebra $C_0(G^{(0)})\subset B\subset C^*_r(G)$ is a closed linear span of normalisers.
	Then there exists an open set $H\subset G$ such that $B=\overline{C_c(H)}$ by Theorem \ref{theorem main theorem about sub module}.
	Since we have the inclusion $C_0(G^{(0)})\subset B$ and $B$ is a C*-subalgebra,
	one can see that $H\subset G$ is an open subgroupoid such that $G^{(0)}\subset H$.
	Now we have shown the corollary.
	\qed
\end{proof}

Using Theorem \ref{theorem main theorem about sub module},
we investigate the inverse semigroups of slices.

\begin{prop}\label{prop slice comes from bisections}
	Let $G$ be an \'etale groupoid and $U\subset G$ be an open set.
	Then $\overline{C_c(U)}\subset C^*_r(G)$ belongs to $\mathcal{S}(C^*_r(G), C_0(G^{(0)}))$ if $U$ is a bisection.
	If $G$ is effective,
	then the converse is true.
\end{prop}

\begin{proof}
	If $U$ is a bisection, it follows that $\overline{C_c(U)}\in \mathcal{S}(C^*_r(G), C_0(G^{(0)}))$ from Proposition \ref{prop open support of normaliser is a bisection}.
	Assume that $G$ is effective, $\overline{C_c(U)}$ is a slice and $\alpha,\beta\in U$ satisfies $d(\alpha)=d(\beta)$.
	Then there exists $f\in C_c(U)$ such that $f(\alpha)=f(\beta)=1$ by Urysohn's lemma.
	Since we assume that $f\in C_c(U)$ is a normalizer,
	$\supp^{\circ}(f)$ is a bisection by Proposition \ref{prop open support of normaliser is a bisection}.
	Hence we have $\alpha=\beta$ and the domain map is injective on $U$.
	Similarly, one can show that the range map is injective on $U$.
	Therefore $U$ is a bisection.
	\qed
\end{proof}

\begin{cor}\label{corollary inverse semigroup of bisections are inverse semigroup of slices}
	Let $G$ be an effective \'etale groupoid.
	Then the map 
	\[\Psi\colon \mathrm{Bis}(G)\ni U\mapsto \overline{C_c(U)}\in \mathcal{S}(C^*_r(G), C_0(G^{(0)})) \]
	in Proposition \ref{prop homomorphism from Bis to slices} is an isomorphism.
\end{cor}

\begin{proof}
	By Proposition \ref{prop homomorphism from Bis to slices}, it is sufficient to show that $\Psi$ is surjective.
	Theorem \ref{theorem main theorem about sub module} and Proposition \ref{prop slice comes from bisections} imply that $\Psi$ is surjective.
	\qed
\end{proof}

We have the following sufficient condition where $C_0(G^{(0)})$-submodule is a closed linear span of normalizers.

\begin{prop}[cf. {\cite[Theorem 3.5(1)]{BrownExelFuller2021}}] \label{prop if coefficient map exists, then bimodule is gen by normalizers}
	Let $G$ be an effective \'etale groupoid and $M\subset C_r^*(G)$ be a closed $C_0(G^{(0)})$-sub-bimodule.
	Assume that there exists a continuous $C_0(G^{(0)})$-bimodule map $\Phi\colon C^*_r(G)\to M$ such that $\Phi|_M=\id$.
	Then $M$ is a closed linear span of normalisers.
\end{prop}

\begin{proof}
	This proof is a simple modification of the one in \cite[Theorem 3.5(1)]{BrownExelFuller2021}.
	Let $n\in C^*_r(G)$ be an intertwiner for $C_0(G^{(0)})$.
	Then $\Phi(n)$ is also an intertwiner for $C_0(G^{(0)})$.
	Indeed, for any $d\in C_0(G^{(0)})$,
	there exists $d'\in C_0(G^{(0)})$ such that $nd=d'n$ since $n$ is an intertwiner.
	Now we have
	\[
	\Phi(n)d=\Phi(nd)=\Phi(d'n)=d'\Phi(n).
	\]
	Therefore $\Phi(n)C_0(G^{(0)})\subset C_0(G^{(0)})\Phi(n)$ holds.
	The same argument implies $\Phi(n)C_0(G^{(0)})\supset C_0(G^{(0)})\Phi(n)$ and we obtain $\Phi(n)C_0(G^{(0)})= C_0(G^{(0)})\Phi(n)$.
	Hence, $\Phi(n)$ is an intertwiner for $C_0(G^{(0)})$.
	Using the continuity of $\Phi$ and Proposition \ref{prop approximation of normalisers by intertwiners}\footnote{Since we assume that $G$ is effective, $C_0(G^{(0)})\subset C^*_r(G)$ is a maximal abelian subalgebra. For example, see \cite[II Proposition 4.7]{renault1980groupoid}.},
	we obtain $\Phi(N(C^*_r(G), C_0(G^{(0)})))\subset N(C^*_r(G),C_0(G^{(0)}))$.
	Now, one can check that $\overline{\Span}\Phi(N(C^*_r(G), C_0(G^{(0)})))=M$.
	Therefore, $M$ is a closed linear span of normalizers.
	\qed
\end{proof}

In the above situation,
the corresponding open set is closed as follows.

\begin{prop}[cf.\ {\cite[Lemma 3.4]{BrownExelFuller2021}}]\label{prop if there exists coeffieicent map, then U is closed}
	Let $G$ be an \'etale groupoid and $U\subset G$ be an open set.
	Assume that there exists a continuous $C_0(G^{(0)})$-module map $\Phi\colon C^*_r(G)\to \overline{C_c(U)}$ such that $\Phi|_{\overline{C_c(U)}}=\id$.
	Then $U$ is closed.
\end{prop}

\begin{proof}
	This proof is essentially same as the one in \cite[Lemma 3.4]{BrownExelFuller2021}.
	Assume that $U$ is not closed and take $\gamma\in \overline{U}\setminus U$.
	Put $M\defeq \overline{C_c(U)}$ and denote the evaluation at $\gamma$ by $\ev_{\gamma}$.
	Namely, $\ev_{\gamma}\in C^*_r(G)^*$ is the map defined by $\ev_{\gamma}(a)=a(\gamma)$ for $a\in C^*_r(G)$ (see Proposition \ref{prop evaluation}).
	Then $\ev_{\gamma}$ is zero on $M$.
	
	Take an open bisection $V\subset G$ with $\gamma\in V$.
	By Urysohn's lemma,
	there exists $f\in C_c(V)$ such that $f(\gamma)=1$.
	Then $\Phi(f)(\gamma)=\ev_{\gamma}(\Phi(f))=0$.
	
	On the other hand, there exists a net $\{\gamma_{n}\}\subset U\cap V$ that converges to $\gamma$.
	For each $\gamma_{n}$,
	take an open set $V_{n}$ such that $\gamma_n\in V_n \subset U\cap V$.
	By Urysohn's lemma,
	take $g_n\in C_c(r(V_n))$ such that $g_n(r(\gamma_n))=1$ for each $n$.
	Then we have $g_n*f\in C_c(V_n) \subset C_c(U)$.
	Now it follows that
	\begin{align*}
	\Phi(f)(\gamma)&=\lim_n \Phi(f)(\gamma_n) =\lim_n g_n*\Phi(f)(\gamma_n) \\
	&=\lim_n \Phi(g_n*f)(\gamma_n)=\lim_ng_n*f(\gamma_n)=1,
	\end{align*}
	which contradicts to $\Phi(f)(\gamma)=0$.
	\qed
\end{proof}

As an application,
we show that discrete group coactions of groupoid C*-algebras come from cocycles if the fixed point algebras contain $C_0(G^{(0)})$.

\begin{cor}\label{corollary coaction of discrete group on groupoid C*-algebras}
	Let $G$ be an effective \'etale groupoid, $\Gamma$ be a discrete group and $\delta\colon C^*_r(G)\to C^*_r(G)\otimes C^*_r(\Gamma)$ be a coaction.
	Assume that the fixed point algebra $C^*_r(G)_e$ contains $C_0(G^{(0)})$.
	Then there exists a continuous cocycle $c\colon G\to \Gamma$ such that $\delta=\delta_c$.
\end{cor}

\begin{proof}
	See Subsection \ref{subsection Discrete group coactions on C*-algebras} for notations of discrete group coactions.
	For $s\in \Gamma$,
	let $C^*_r(G)_s$ denote the spectral module for $s\in\Gamma$.
	Then $C^*_r(G)_s$ is a $C^*_r(G)_e$-sub-bimodule.
	Since we assume that $C^*_r(G)_e$ contains $C_0(G^{(0)})$,
	$C^*_r(G)_s$ is a $C_0(G^{(0)})$-sub-bimodule.
	Moreover,
	$\Phi_s\colon C^*_r(G)\to C^*_r(G)_s$ is a $C_0(G^{(0)})$-bimodule map.
	By Proposition \ref{prop if coefficient map exists, then bimodule is gen by normalizers},
	$C^*_r(G)_s$ is a closed linear span of normalizers.
	Hence,
	by Theorem \ref{theorem main theorem about sub module} and Proposition \ref{prop if there exists coeffieicent map, then U is closed},
	there exists an open and closed set $U_s\subset G$ such that $C^*_r(G)_s=\overline{C_c(U_s)}$.
	Since $C^*_r(G)_s\cap C^*_r(G)_t=\{0\}$ holds for $s,t\in \Gamma$ with $s\not=t$,
	a family $\{U_s\}_{s\in\Gamma}$ is pairwise disjoint.
	In addition, since $C^*_r(G)=\overline{\Span}\bigcup_{s\in\Gamma}C^*_r(G)_s$,
	we have $G=\bigcup_{s\in\Gamma}U_s$.
	Therefore,
	$\{U_s\}_{s\in\Gamma}$ is a partition of $G$.
	Accordingly, we may define $c\colon G\to \Gamma$ so that $U_s=c^{-1}(\{s\})$ holds for all $s\in\Gamma$.
	Then $c\colon G\to \Gamma$ is continuous since each $U_s$ is open.
	One can check that $c\colon G\to \Gamma$ is cocycle.
	Hence, we obtain the continuous cocycle $c\colon G\to\Gamma$.
	
	Now, it remains to show that $\delta=\delta_c$.
	It is clear that $\delta(a)=\delta_c(a)$ holds for all $a\in \Span \bigcup_{s\in\Gamma}C^*_r(G)_s$ by the definition of $\delta_c$ and the equation $C^*_r(G)_s=\overline{C_c(U_s)}$.
	Since $\Span \bigcup_{s\in\Gamma}C^*_r(G)_s$ is dense in $C^*_r(G)$,
	we obtain $\delta=\delta_c$
	\qed
\end{proof}

\subsection{Galois correspondence for discrete group coactions on groupoid C*-algebras}\label{subsection Galois correspondence for discrete group coactions on groupoid C*-algebras }

In this subsection,
we investigate open subgroupoids which contain the kernels of cocycles.
Precisely, we prove the next theorem in this subsection.

\begin{thm}\label{theorem intermediate subgroupoids and subgroups}
	Let $G$ be an \'etale groupoid, $\Gamma$ be a discrete group and $c\colon G\to \Gamma$ be a continuous cocycle.
	Assume that $\ker c\subset G$ is minimal and $c\colon G\to \Gamma$ is surjective.
	Set 
	\[\mathcal{D}\defeq\{H\subset G\mid \ker c\subset H,  \text{$H$ is an open subgroupoid.}\} \]
	and
	\[\mathcal{E}\defeq\{\Lambda\subset\Gamma\mid \text{$\Lambda$ is a subgroup.}\}.\]
	Then the map $\mathcal{D}\ni H\mapsto c(H)\in\mathcal{E}$ is bijective and the inverse map is given by $\Lambda\mapsto c^{-1}(\Lambda)$.
	In particular,
	an open subgroupoid $H\in\mathcal{D}$ is automatically closed.
\end{thm}

We divide a proof of Theorem \ref{theorem intermediate subgroupoids and subgroups} into a few lemmas.
First, we show that the map in Theorem \ref{theorem intermediate subgroupoids and subgroups} is well-defined.

\begin{lem}
	Let $G$ be an \'etale groupoid, $\Gamma$ be a discrete group and $c\colon G\to \Gamma$ be a continuous cocycle.
	Assume that $\ker c$ is topologically transitive.
	For $H\in\mathcal{D}$, $c(H)\subset \Gamma$ is a subgroup.
\end{lem}
\begin{proof}
	
	It suffices to show $st\in c(H)$ for all $s,t\in c(H)$,
	since it is clear that $c(H)$ is closed under the inverse.
	Take $\alpha,\beta\in H$ so that $s=c(\alpha)$ and $t=c(\beta)$ hold.
	Since $d(c^{-1}(\{s\} )\cap H)$ and $r(c^{-1}(\{t\})\cap H)$ are  non-empty open sets and $\ker c$ is topologically transitive,
	there exists $\gamma\in \ker c$ such that $d(\gamma)\in r(c^{-1}(\{t\})\cap H)$ and $r(\gamma)\in d(c^{-1}(\{s\})\cap H)$ hold.
	Take $\alpha'\in c^{-1}(\{s\})\cap H$ and $\beta'\in c^{-1}(\{t\})\cap H$ such that $d(\alpha')=r(\gamma)$ and $r(\beta')=d(\gamma)$.
	Now we have
	\[
	st=c(\alpha')c(\gamma)c(\beta')=c(\alpha'\gamma\beta')\in c(H).
	\]
	Thus, $c(H)$ is a subgroup of $\Gamma$.
	\qed
\end{proof}

\begin{rem}
We remark that the images of cocycles do not necessarily become subgroups in general.
Indeed, put $G\defeq\Z/2\Z\coprod \Z/2\Z$ and $\Gamma \defeq \Z/2\Z \times\Z/2\Z$.
Then $G$ is a discrete, hence \'etale, groupoid such that $\lvert G^{(0)}\rvert=2$ and $d=r$ hold.
There exists a (continuous) cocycle $c\colon G\to\Gamma$ such that $c(G)=\{(0,0), (1,0), (0,1)\}$,
which is not a subgroup of $\Gamma$. 
\end{rem}

The next lemma follows from the elementary set theory.

\begin{lem}
	Let $G$ be an \'etale groupoid, $\Gamma$ be a discrete group and $c\colon G\to \Gamma$ be a continuous cocycle.
	For $\Lambda\in \mathcal{E}$,
	$c(c^{-1}(\Lambda))=\Lambda$ holds if $c$ is surjective.
\end{lem}

Finally, we complete the proof of Theorem \ref{theorem intermediate subgroupoids and subgroups} by showing the next lemma.

\begin{lem}
	Let $G$ be an \'etale groupoid, $\Gamma$ be a discrete group and $c\colon G\to \Gamma$ be a continuous cocycle.
	Then $c^{-1}(c(H))=H$ holds for $H\in\mathcal{D}$ if $\ker c\subset G$ is minimal.
\end{lem}
\begin{proof}
	For $H\in\mathcal{D}$,
	the inclusion $H\subset c^{-1}(c(H))$ is obvious.
	We show  $H\supset c^{-1}(c(H))$.
	Take $\gamma\in c^{-1}(c(H))$.
	Then there exists $\alpha\in H$ such that $c(\gamma)=c(\alpha)$.
	Since $r(c^{-1}(\{c(\alpha)\})\cap H)\subset G^{(0)}$ is a non-empty open set and $\ker c$ is minimal,
	there exists $\beta\in \ker c$ such that $d(\beta)=r(\gamma)$ and $r(\beta)\in r(c^{-1}(\{c(\alpha)\})\cap H)$ hold.
	In addition, there exists $\alpha'\in c^{-1}(\{c(\alpha)\})\cap H$ such that $r(\alpha')=r(\beta)$.
	Now $\alpha'^{-1},\beta,\gamma$ are composable and $c(\alpha'^{-1}\beta\gamma)=e$ holds.
	Since we have $\beta, \alpha'^{-1}\beta\gamma\in\ker c\subset H$ and $\alpha'\in H$,
	it follows 
	\[\gamma=\beta^{-1}\alpha'(\alpha'^{-1}\beta\gamma)\in H.\]
	Now we have shown $c^{-1}(c(H))=H$.
	\qed
\end{proof}

\begin{ex}\label{example subgroupoid of Deaconu-Renault system}
	We investigate open subgroupoids of Deaconu-Renault groupoids.
	See \cite[Example 2.3.7]{asims} for Deaconu-Renault groupoids.
	Let $X$ be a locally compact Hausdorff space and $T\colon X\to X$ be a local homeomorphism.
	Then the Deaconu-Renault groupoid $G(X,T)$ is defined as follows.
	First, 
	put
	\begin{align*}
	&G(X,T)\defeq \{(y, n-m,x)\in X\times \Z \times X\mid T^n(y)=T^m(x), n,m\in\N\}, \\
	&G(X,T)^{(0)}\defeq \{(x,0,x)\in G(X,T)\mid x\in X\}.
	\end{align*}
	We identify $G(X,T)^{(0)}$ with $X$ in the obvious way.
	Define $r((y,n,x))=y$ and $d((y,n,x))=x$.
	For $(z,n,y),(y,m,x)\in G(X,T)$,
	their product is $(z, n+m, x)\in G(X,T)$.
	In these operations,
	$G(X,T)$ is a groupoid.
	Remark that the inverse is given by $(y,n,x)^{-1}=(x,-n, y)$.
	For open sets $U,V\subset X$ and $n,m\in\N$,
	put
	\[
	Z(U,n,m,V)\defeq \{(y,n-m,x)\in G(X,T)\mid T^n(y)=T^m(x), y\in U, x\in V\}.
	\]
	A family of these sets $Z(U,n,m,V)$ becomes an open basis of $G(X,T)$.
	With respect to this topology,
	$G(X,T)$ is an \'etale locally compact Hausdorff groupoid.
	
	Define a map $c\colon G(X,T)\to \Z$ by $c((y,n,x))=n$.
	Then $c$ is a continuous cocycle.
	If $\ker c$ is minimal,
	we may apply Theorem \ref{theorem intermediate subgroupoids and subgroups}.
	Then it follows that a proper intermediate open subgroupoid $\ker c\subset H\subset G(X,T)$ is of the form
	\[
	H=\{(y,n-m,x)\in G(X,T)\mid n-m\in k\Z\}
	\]
	for some $k\in \N$.
	For example,
	let $\Sigma$ be a finite set, $X\defeq \Sigma^{\N}$ and $\sigma\colon X\to X$ be the shift map (i.e.\ $\sigma(\{x_i\}_{i\in\N})=\{x_{i+1}\}_{i\in\N}$ for $\{x_i\}\in X$).
	Then the kernel of $c\colon G(X,\sigma)\to\Z$ is minimal (we will observe this minimality in Example \ref{example subalgebra of Cuntz algebra On}).
	This example gives a way to compute certain intermediate subalgebras of the Cuntz algebras $\mathcal{O}_n$.
	See Example \ref{example subalgebra of Cuntz algebra On}.
	
\end{ex}

\begin{ex}\label{example counter example of infinite size Cuntz}
	In Theorem \ref{theorem intermediate subgroupoids and subgroups},
	we assumed that the kernels of cocycles are minimal.
	If we replace ``minimal'' with ``topologically transitive'',
	then Theorem \ref{theorem intermediate subgroupoids and subgroups} does not hold.
	We give such an example here.
	We use the universal groupoid of the polycyclic monoid $P_{\infty}$.
	See \cite[Example 3, p182]{paterson2012groupoids} for more details.
	
	Let $\Sigma$ be a countably infinite set and $\Sigma^*\defeq \bigcup_{n\in\N}\Sigma^n$ be the set of finite words.
	In addition, put $X\defeq \Sigma^*\cup \Sigma^{\N}$.
	For $\mu\in\Sigma^*$ and a finite set $F\subset \Sigma^*$,
	define $C_{\mu,F}\subset X$ to be the set of finite or infinite words that begin with $\mu$ and do not begin with the elements in $F$.
	If $F=\emptyset$,
	we simply write $C_{\mu}=C_{\mu,F}$. 
	Then a family of all $C_{\mu, F}$ is an open basis of $X$.
	With respect to this topology,
	$X$ is a compact Hausdorff set.
	Let $P_{\infty}$ be the polycyclic monoid of infinite size.
	Namely,
	$P_{\infty}$ is the universal inverse semigroup generated by $\{s_a\}_{a\in\Sigma}\cup \{0,1\}$ such that $s_a^*s_b=\delta_{a,b}1$ for all $a,b\in\Sigma$.
	For $\mu\in\Sigma^n$,
	put $s_{\mu}\defeq s_{\mu_1}s_{\mu_2}\cdots s_{\mu_{n}}\in P_{\infty}$.\footnote{For the empty word $\emptyset$,
	define $s_{\emptyset}=1$.}
	Then we have the equation
	\[
	P_{\infty}=\{s_{\mu}s_{\nu}^*\mid \mu,\nu\in\Sigma^*\}\cup\{0\}.
	\]
	
	We define an action $\alpha\colon P_{\infty}\curvearrowright X$ as follows.
	For $s_{\mu}s_{\nu}^*$,
	define $\alpha_{s_{\mu}s_{\nu}^*}\colon C_{\nu}\to C_{\mu}$ by $\alpha_{s_{\mu}s_{\nu}^*}(\nu x)=\mu x$ for $x\in X$.
	Then $\alpha_{s_{\mu}s_{\nu}^*}$ is a homeomorphism and the map $s_{\mu}s_{\nu}^*\mapsto \alpha_{s_{\mu}s_{\nu}^*}$ defines an action $\alpha\colon P_{\infty}\curvearrowright X$.
	
	Put $G\defeq P_{\infty}\ltimes_{\alpha}X$ and define a continuous cocycle $c\colon G\to\Z$ by 
	\[c([s_{\mu}s_{\nu}^*, x])=\lvert\mu\rvert-\lvert\nu\rvert,\]
	where $\lvert \mu\rvert$ denotes the length of $\mu$.
	It follows from \cite[Theorem 3.5]{Paterson2002} that $G$ is Hausdorff.
	For the reader's convenience,
	we check that $G$ is Hausdorff using $\ker c$.
	Take $\gamma \in \overline{G^{(0)}}\subset G$.
	Since $\ker c$ is closed in $G$ and $G^{(0)}\subset \ker c$,
	we have $\gamma\in\ker c$.
	Therefore there exists $l\in \N$, $\mu,\nu\in \Sigma^l$ and $\eta\in X$ such that
	\[
	\gamma = [s_{\mu}s_{\nu}^*, \nu\eta].
	\]
	Since $\gamma\in\overline{G^{(0)}}$,
	the domain and range of $\gamma$ coincide and we have
	\[\nu\eta=d(\gamma)=r(\gamma)=\mu\eta.\]
	Hence we obtain $\mu=\nu$ and therefore $\gamma\in G^{(0)}$.
	Since $G^{(0)}$ turns out to be closed in $G$,
	$G$ is Hausdorff by \cite[Lemma 8.3.2]{asims}.
	
	In addition, we observe that $\ker c$ is not minimal but topologically transitive.
	Indeed, since the orbit of the empty word $\emptyset$ in $\ker c$ is $\{\emptyset\}$ itself,
	$\ker c$ is not minimal.
	Next, we show that $\ker c$ is topologically transitive.
	Fix $a\in \Sigma$ and put $x\defeq aaa\cdots \in \Sigma^{\N}$.
	We claim that the orbit of $x$ in $\ker c$ is dense in $X$.
	Take $\mu\in\Sigma^*$ and a finite set $F\subset \Sigma^*$ so that $C_{\mu, F}$ is not empty.
	Replacing $F$ appropriately,
	we may assume that for every $\nu\in F$ there exists a $\eta_{\nu}\in\Sigma^*\setminus\{\emptyset\}$ such that $\nu=\mu\eta_{\nu}$.
	Then there exists $w\in \Sigma^*$ such that the prefixes of $w$ do not coincide with $\eta_{\nu}$ for all $\nu\in F$ since $\Sigma$ is an infinite set.
	Putting $y=\mu w aaa\cdots\in\Sigma^{\N}$,
	one can see that $y\in C_{\mu, F}$ and 
	\[\alpha_{s_{\mu w}s^*_{v}}(x)=y\]
	holds,
	where we put $v\defeq aa\cdots a\in \Sigma^{\lvert\mu\rvert +\lvert w\rvert}$ here.
	Since we have $s_{\mu w}s^*_{v}\in \ker c$,
	$y$ is contained in the orbit of $x$ in $\ker c$.
	Therefore $\ker c$ is topologically transitive since $\ker c$ has a dense orbit.

	Put
	\[
	T\defeq \{s_{\mu}s_{\nu}^*\in P_{\infty}\mid \lvert \mu\rvert -\lvert \nu\rvert\in 2\Z, \lvert\mu\rvert\geq 1, \lvert\nu\rvert\geq 1\}\cup \{0,1\}.
	\]
	Then $T$ is an inverse subsemigroup of $P_{\infty}$.
	Note that $T\ltimes_{\alpha}X$ can be regarded as an open subgroupoid of $G$ in the obvious way.
	One can see that $c(T\ltimes_{\alpha}X)=2\Z$ and $\ker c\subset T\ltimes_{\alpha}X\subsetneq c^{-1}(2\Z)$ holds.
	Indeed, 
	put $\gamma\defeq [s_a^2, \emptyset]\in c^{-1}(2\Z)$,
	where $a\in\Sigma$ is an arbitrary element and $\emptyset\in\Sigma^*$ is the empty word.
	Then one can check that$\gamma\not\in T\ltimes_{\alpha}X$.
	Therefore, the map $\mathcal{D}\ni H\mapsto c(H)\in\mathcal{E}$ in Theorem \ref{theorem intermediate subgroupoids and subgroups} is not injective and, in particular, Theorem \ref{theorem intermediate subgroupoids and subgroups} does not hold for this example.
	
	We remark that $C^*_r(P_{\infty}\ltimes_{\alpha}X)$ is isomorphic to the Cuntz algebra $\mathcal{O}_{\infty}$ (see \cite[Corollary 3.9]{Paterson2002}).
	The above \'etale groupoid $T\ltimes_{\alpha}X$ yields a strange intermediate C*-subalgebra $\mathcal{O}^{\T}_{\infty}\subset B\subset \mathcal{O}_{\infty}$ (see Example \ref{example strange intermediate subalgebra}).
\end{ex}

We apply our results to the theory of groupoid C*-algebras.
Summarizing Corollary \ref{corollary Galois corr between subgroupoid and subalgebra}, Corollary \ref{corollary coaction of discrete group on groupoid C*-algebras} and Theorem \ref{theorem intermediate subgroupoids and subgroups},
we obtain the next corollary.

\begin{cor}\label{corollary galois corr for discrete group coactions}
	Let $G$ be an effective groupoid, $\Gamma$ be a discrete group and $\delta\colon C^*_r(G)\to C^*_r(G)\otimes C^*_r(\Gamma)$ be a coaction.
	Assume that the fixed point algebra $C^*_r(G)_e$ is simple and contains $C_0(G^{(0)})$.
	Then the map $\Lambda\mapsto C^*_r(c^{-1}(\Lambda))$ is a bijection between the set of subgroups $\Lambda\subset \Gamma$ and the set of intermediate C*-subalgebras $C^*_r(G)_e\subset B\subset C^*_r(G)$ which are closed linear spans of normalisers.
	In particular,
	if an intermediate C*-subalgebra $C^*_r(G)_e\subset B\subset C^*_r(G)$ is a closed linear span of normalisers,
	then there exists a conditional expectation $E\colon C^*_r(G)\to B$.
\end{cor}

\begin{proof}
	By Corollary \ref{corollary coaction of discrete group on groupoid C*-algebras},
	there exists a continuous cocycle $c\colon G\to \Gamma$ such that $\delta=\delta_c$.
	Since we assume that $C^*_r(G)_e=C^*_r(\ker c)$ is simple,
	$\ker c$ is minimal (for example, see \cite[Proposition 5.7]{Brown2014}).
	Now, combining Corollary \ref{corollary Galois corr between subgroupoid and subalgebra} with Theorem \ref{theorem intermediate subgroupoids and subgroups},
	it follows that the map $\Lambda\mapsto C^*_r(c^{-1}(\Lambda))$ in the statement is bijective.
	
	The last assertion follows from the fact that for an open subgroupoid $G^{(0)}\subset H\subset G$, there exists a conditional expectation $E\colon C^*_r(G)\to C^*_r(H)$ if and only if $H\subset G$ is closed (see \cite[Lemma 3.4]{BrownExelFuller2021}).
	\qed
\end{proof}

\begin{ex}\label{example subalgebra of Cuntz algebra On}
	Using Example \ref{example subgroupoid of Deaconu-Renault system},
	we investigate intermediate C*-subalgebras of the inclusion $\mathcal{O}_n^{\T}\subset \mathcal{O}_n$,
	where $\mathcal{O}_n$ is the Cuntz algebra and $\mathcal{O}_n^{\T}$ is the fixed point algebra of the gauge action $\T\curvearrowright \mathcal{O}_n$.
	Let $\Sigma$ be a finite set with $\lvert\Sigma\rvert=n$, $X\defeq \Sigma^{\N}$ and $\sigma\colon X\to X$ be the shift map.
	Let $G\defeq G(X, \sigma)$ denote the Deaconu-Renault groupoid (see Example \ref{example subgroupoid of Deaconu-Renault system}).
	Then $C^*_r(G(X,\sigma))$ is isomorphic to the Cuntz algebra $\mathcal{O}_n$.
	Indeed,
	for $a\in \Sigma$,
	denote the set of all infinite words beginning with $a$ by $C_a$.
	Then $Z(X, 1,0, C_a)$ is a compact open set of $G$ and the characteristic function $S_a\defeq \chi_{Z(X, 1,0, C_a)}$ belongs to $C_c(G)$.
	One can see that $\{S_a\}_{a\in\Sigma}$ satisfies the Cuntz relation.
	In addition, $\{S_a\}_{a\in\Sigma}$ generates $C^*_r(G)$.
	Hence $C^*_r(G)$ is isomorphic to the Cuntz algebra $\mathcal{O}_n$.
	From now on, we identify $C^*_r(G)$ with $\mathcal{O}_n$.
	
	Consider the cocycle $c\colon G\to\Z$ defined by $c((y,m,x))=m$.
	Then $\ker c$ is minimal.
	Indeed,
	take $x\in X$, $l\in\N$ and $\mu\in \Sigma^l$ arbitrarily.
	Then we have
	\[
	(\mu x_{l+1}x_{l+2}\cdots,0,x )\in \ker c
	\]
	and therefore the orbit of $x$ in $\ker c$ intersects with $C_{\mu}$,
	where $C_{\mu}$ denotes the set of all infinite words beginning with $\mu$.
	Since $\{C_{\mu}\}_{\mu\in\Sigma^*}$ is a basis of $X$,
	the orbit of $x$ in $\ker c$ is dense in $X$ and therefore $\ker c$ is minimal.
	Hence, an intermediate C*-subalgebra $C^*_r(G)_e\subset B\subset C^*_r(G)$ that is a closed linear span of normalisers is of the form $C^*_r(c^{-1}(k\Z))$ for some $k\in\N$.
	
	We remark that the fixed point algebra $C^*_r(G)_e$ of the coaction $\delta_c$ coincides with the fixed point algebra $\mathcal{O}_n^{\T}$ of the gauge action $\T\curvearrowright \mathcal{O}_n$.
	In addition,
	one can see that 
	\[
	C^*_r(c^{-1}(k\Z))=\overline{\Span}\{S_{\mu}S_{\nu}^*\in\mathcal{O}_n\mid \mu,\nu\in\Sigma^*, \lvert\mu\rvert-\lvert\nu\rvert\in k\Z \}
	\]
	holds, where we define $S_{\mu}\defeq S_{\mu_1}S_{\mu_2}\cdots S_{\mu_{\lvert\mu\rvert}}$ for $\mu\in\Sigma^*$.
	Hence, it turns out that an intermediate C*-subalgebras $\mathcal{O}_n^{\T}\subset B \subset\mathcal{O}_n$ is of the form
	\[
	B=\overline{\Span}\{S_{\mu}S_{\nu}^*\in\mathcal{O}_n\mid \mu,\nu\in\Sigma^*, \lvert\mu\rvert-\lvert\nu\rvert\in k\Z \}
	\]
	for some $k\in \N$ if $B$ is a closed linear span of normalisers.
\end{ex}

\begin{ex}\label{example strange intermediate subalgebra}
	We investigate intermediate C*-subalgebras between the inclusion $\mathcal{O}_{\infty}^{\T}\subset \mathcal{O}_{\infty}$,
	where $\mathcal{O}_{\infty}^{\T}$ is the fixed point algebra of the gauge action $\T\curvearrowright \mathcal{O}_{\infty}$.
	In case of $\mathcal{O}_{\infty}$,
	there exist more intermediate C*-subalgebras between the inclusion $\mathcal{O}_{\infty}^{\T}\subset \mathcal{O}_{\infty}$ than the ones between the inclusion $\mathcal{O}^{\T}_n\subset \mathcal{O}_n$.
	Inspired by Example \ref{example counter example of infinite size Cuntz},
	we construct a somehow strange intermediate C*-subalgebra $\mathcal{O}_{\infty}^{\T}\subset B\subset \mathcal{O}_{\infty}$.
	Precisely, there exists an intermediate C*-subalgebra $\mathcal{O}^{\T}_{\infty}\subset B \subset\mathcal{O}_{\infty}$ that is not of the form
	\[
	\overline{\Span}\{S_{\mu}S_{\nu}^*\in\mathcal{O}_{\infty}\mid \mu,\nu\in\Sigma^*, \lvert\mu\rvert-\lvert\nu\rvert\in k\Z \}
	\]
	for any $k\in\N$,
	where $\Sigma$ denotes a countably infinite set and $\{S_a\}_{a\in\Sigma}$ denotes the generators of $\mathcal{O}_{\infty}$.
	To see this, define \[B\defeq \overline{\Span}\big( \{S_{\mu}S_{\nu}^*\in \mathcal{O}_{\infty}\mid \mu,\nu\in\Sigma^*, \lvert\mu\rvert-\lvert\nu\rvert\in 2\Z, \lvert\mu\rvert\lvert\nu\rvert\geq 1 \}\cup \{1\}\big). \]
	We claim that $B$ is a C*-subalgebra $\mathcal{O}_{\infty}^{\T}\subset B \subset \mathcal{O}_{\infty}$ which is not of the form 
	\[
	\overline{\Span}\{S_{\mu}S_{\nu}^*\in\mathcal{O}_{\infty}\mid \mu,\nu\in\Sigma^*, \lvert\mu\rvert-\lvert\nu\rvert\in k\Z \}
	\]
	for any $k\in\N$.
	It is easy to check that $\mathcal{O}_{\infty}^{\T}\subset B \subset \mathcal{O}_{\infty}$ is an intermediate C*-subalgebra.
	We show that $B$ dose not contain $S_a^k$ for any $a\in\Sigma$ and $k\in\N_{>0}$,
	which implies that $B$ is not of the form
	\[\overline{\Span}\{S_{\mu}S_{\nu}^*\in\mathcal{O}_{\infty}\mid \mu,\nu\in\Sigma^*, \lvert\mu\rvert-\lvert\nu\rvert\in k\Z \}\]
	for any $k\in\N$.
	For $S_{\mu}S_{\nu}^*\in\mathcal{O}_{\infty}$,
	define a bounded linear operator $\pi(S_{\mu}S_{\nu}^*)\in B(\ell^2(\Sigma^*))$ by
	\begin{align*}
	\pi(S_{\mu}S_{\nu}^*)\delta_{\eta}=
	\begin{cases}
		\delta_{\mu\rho} & (\text{$\eta=\nu\rho$ for some $\rho\in\Sigma^*$} ) \\
		0 & (\text{otherwise}).
	\end{cases}
	\end{align*}
	Then we obtain the *-representation $\pi\colon\mathcal{O}_{\infty}\to B(\ell^2(\Sigma^*))$.
	Define a functional $\varphi_{a,k}\in\mathcal{O}_{\infty}^*$ by
	\[
	\varphi_{a,k}(x)\defeq \i<\pi(S_a^k)\delta_{\emptyset}| \pi(x)\delta_{\emptyset}>,
	\]
	where $x\in \mathcal{O}_{\infty}$.
	Then we have $B\subset \ker\varphi_{a,k}$ and $\varphi_{a,k}(S_a^k)=1$.
	Hence we obtain $S_a^k\not\in B$ for any $a\in\Sigma$ and $k\in\N_{>0}$ and $B$ is not of the form
	\[\overline{\Span}\{S_{\mu}S_{\nu}^*\in\mathcal{O}_{\infty}\mid \mu,\nu\in\Sigma^*, \lvert\mu\rvert-\lvert\nu\rvert\in k\Z \}\]
	for any $k\in\N$.
	
\end{ex}

\bibliographystyle{plain}
\bibliography{bunken}

\end{document}